\documentclass[UTF-8,reqno]{amsart}
\usepackage{enumerate}
\usepackage{amssymb,url,color, booktabs}

\usepackage{mathrsfs}
\usepackage{amsmath}
\usepackage{verbatim}
\usepackage{fancyhdr}
\usepackage[nobysame]{amsrefs}
\BibSpec{article}{%
+{}{\PrintAuthors} {author}
+{,}{ \textrm} {title}
+{.}{ \textit} {journal}
+{,}{ \textbf} {volume}
+{}{ \parenthesize} {date}
+{,}{ } {pages}
+{.}{ arXiv:} {eprint}
+{.}{} {transition}
}
\BibSpec{book}{%
+{}{\PrintAuthors} {author}
+{,}{ \textit} {title}
+{.}{ \textrm} {series} 
+{,}{ Vol.} {volume}
+{.}{ } {publisher}
+{,}{ } {date}
+{.}{} {transition}
}
\usepackage{color}
\usepackage[colorlinks=true]{hyperref}
\hypersetup{
    linkcolor=blue,          
    citecolor=red,        
    filecolor=blue,      
    urlcolor=cyan
}

\numberwithin{equation}{section}

\newcommand{\be}{\begin{eqnarray}}
\newcommand{\ee}{\end{eqnarray}}
\newcommand{\ce}{\begin{eqnarray*}}
\newcommand{\de}{\end{eqnarray*}}
\newtheorem{theorem}{Theorem}[section]
\newtheorem{lemma}[theorem]{Lemma}
\newtheorem{remark}[theorem]{Remark}
\newtheorem{definition}[theorem]{Definition}
\newtheorem{proposition}[theorem]{Proposition}
\newtheorem{Examples}[theorem]{Example}
\newtheorem{corollary}[theorem]{Corollary}

\newenvironment{proof of theorem 1.2 and 1.3}{{\it Proof of Theorem 1.2 and 1.3}.}{{\hfill 	
$\square$\hskip - \parfillskip}}
\newenvironment{proof of theorem 1.4}{{\it Proof of Theorem 1.4}.}{{\hfill 	
		$\square$\hskip - \parfillskip}}
\newenvironment{proof of theorem 1.5}{{\it Proof of Theorem 1.5}.}{{\hfill 	
		$\square$\hskip - \parfillskip}}
\newenvironment{proof of theorem 1.6}{{\it Proof of Theorem 1.6}.}{{\hfill 	
		$\square$\hskip - \parfillskip}}

\makeatletter

\newcommand{\Rmnum}[1]{\expandafter\@slowromancap\romannumeral #1@}
\makeatother

\def\[{{\Big[}}
\def\]{{\Big]}}
\def\<{{\langle}}
\def\>{{\rangle}}
\def\({{\Big(}}
\def\){{\Big)}}

\def\bx{{\mathbf{x}}}

\def\={&\!\!=\!\!&}

\def\1{{\mathbf{1}}}

\def\geq{\geqslant}
\def\leq{\leqslant}

\def\k{\kappa}

\def\[{{\Big[}}
\def\]{{\Big]}}
\def\<{{\langle}}
\def\>{{\rangle}}
\def\({{\Big(}}
\def\){{\Big)}}

\def\bx{{\mathbf{x}}}

\def\W{{\mathcal W}}

\def\={&\!\!=\!\!&}
\def\bt{\begin{theorem}}
\def\et{\end{theorem}}
\def\bl{\begin{lemma}}
\def\el{\end{lemma}}
\def\br{\begin{remark}}
\def\er{\end{remark}}
\def\bx{\begin{Examples}}
\def\ex{\end{Examples}}
\def\bd{\begin{definition}}
\def\ed{\end{definition}}
\def\bp{\begin{proposition}}
\def\ep{\end{proposition}}
\def\bc{\begin{corollary}}
\def\ec{\end{corollary}}

\def\geq{\geqslant}
\def\leq{\leqslant}

 \def\R{\mathbb R}
 \def\R{\mathbb R}

\def\<{\langle} \def\>{\rangle}

\def\bpf{\begin{proof}}
\def\epf{\end{proof}}

\allowdisplaybreaks

\begin{document}
	
\title{Generalized  Minkowski Formulas and rigidity results for Anisotropic Capillary Hypersurfaces}\thanks{\it {The research is partially supported by NSFC (Nos. 11871053 and 12261105).}}
\author{Jinyu Gao and Guanghan Li}

\thanks{{\it 2020 Mathematics Subject Classification. 53A10, 53C24, 53C40.}}
\thanks{{\it Keywords: Hsiung-Minkowski formula, anisotropic capillary hypersurface,  Alexandrov type theorem,  Minkowski problem}}
\thanks{Email address: jinyugao@whu.edu.cn; ghli@whu.edu.cn}

\address{School of Mathematics and Statistics, Wuhan University, Wuhan 430072, China.
}

\begin{abstract}
In this paper, we obtain a new Hsiung-Minkowski integral formula for anisotropic capillary hypersurfaces in the half-space, which includes the weighted Hsiung-Minkowski formula and classical anisotropic Minkowski identity for closed hypersurfaces as special cases. As applications, we prove some anisotropic Alexandrov-type theorems and rigidity results for  anisotropic capillary hypersurfaces. Specially, the uniqueness of the solution to the anisotropic Orlicz-Christoffel-Minkowski problem is obtained, and thus a new proof is provided for the uniqueness of the solution to $L_p$-Minkowski problem with $p\geq 1$ in the Euclidean capillary convex bodies geometry.
\end{abstract}

\maketitle
\setcounter{tocdepth}{2}
\tableofcontents

\section{Introduction}
Capillary phenomena appears in the study of the equilibrium state of liquid fluids. The study of capillary geometry from the view of mathematics has attracted a lot of attention in the last few decades. For the detailed description of isotropic and anisotropic capillary hypersurfaces, we refer to \cite{Finn1986,Philippis2015,Jia-Wang-Xia-Zhang2023,Koiso2022}.

 The anisotropic geometry is a   great tool for dealing with convex geometry problems, such as Alexandrov-Fenchel inequalities (see \cite{Xia2017,Gao-Li2024,Wei-Xiong2021,Wei-Xiong2022}, and so on). In this paper, we  derive the generalization of anisotropic Hsiung-Minkowski integral formula  and inequalities (Theorem \ref{lemma:hminkow-w0} and Corollary \ref{cor:minkowski neq}), then deduce some anisotropic rigidity results (Theorem \ref{Thm-Alexan-a=b-w0}, \ref{thm1.9}$\thicksim$\ref{thm1.10}). By anisotropic  rigidity result (Theorem \ref{Thm-Alexan-a=b-w0}), we provide a new proof for the uniqueness of the solution to $L_p$-Minkowski problem in capillary convex bodies geometry (Corollary \ref{cor:isotropic-Minkowski}).

This article mainly considers the anisotropic capillary hypersurfaces in the half-space
\begin{align*}
	\mathbb{R}_+^{n+1}=\{x\in \mathbb{R}^{n+1} :\<x,E_{n+1}\>>0 \},
\end{align*}
where $E_{n+1}$ denotes the $(n+1)$th-coordinate unit vector.

 Let $\mathcal{W}\subset \mathbb{R}^{n+1}$ be a given smooth closed strictly convex hypersurface containing the origin with support function $F:\mathbb{S}^{n}\rightarrow\mathbb{R}_+$. The Cahn-Hoffman map associating with $F$ is given by
 \begin{align*}
 \Phi:\mathbb{S}^n\rightarrow\mathbb{R}^{n+1},\quad	\Phi(x)=F(x)x+\nabla^{\mathbb{S}} {F}(x),
 \end{align*}
where $\nabla^{\mathbb{S}}$ denotes the covariant derivative on $\mathbb{S}^n$ with standard round metric. Let $\Sigma$ be a $C^2$ compact orientable hypersurface in $\overline{\mathbb{R}_+^{n+1}}$ with boundary $\partial\Sigma\subset\partial\mathbb{R}_+^{n+1}$, which encloses a bounded domain $\Omega$. Let $\nu$ be the unit normal of $\Sigma$ pointing outward of $\Omega$. Given a constant $\omega_0 \in(-F(E_{n+1}), F(-E_{n+1}))$, we say $\Sigma$ is an anisotropic $\omega_0$-capillary hypersurface (see \cite{Jia-Wang-Xia-Zhang2023})  if
\begin{align}
	\label{equ:w0-capillary}
	\<\Phi(\nu),-E_{n+1}\>=\omega_0,\quad \text{on}\ \partial\Sigma.
\end{align}
If $\W=\mathbb{S}^{n}$, this is just the isotropic case: $\theta$-capillary hypersurface (see e.g., \cite{Wang-Weng-Xia,Xia-arxiv}), i.e. $\Sigma\in\overline{\mathbb{R}^{n+1}_+}$ with a given constant $\theta\in (0,\pi)$ satisfying
\begin{align*}
	\<\nu,E_{n+1}\>=\cos\theta,\quad\text{along}\ \partial\Sigma.
\end{align*}
 The model example of anisotropic $\omega_0$-capillary hypersurface is the $\omega_0$-capillary Wulff shape $\W_{r_0,\omega_0}$ defined by \eqref{equ:crwo}, which is a part of a Wulff shape in $\overline{\mathbb{R}^{n+1}_+}$ such that the anisotropic capillary boundary condition \eqref{equ:w0-capillary} holds \cite{Jia-Wang-Xia-Zhang2023}.

For the convenience of description, we need a constant vector $E_{n+1}^F \in \mathbb{R}^{n+1}$ defined as
 $$
 E_{n+1}^F= \begin{cases}\frac{\Phi\left(E_{n+1}\right)}{F\left(E_{n+1}\right)}, & \text { if } \omega_0<0, \\[5pt]
  -\frac{\Phi\left(-E_{n+1}\right)}{F\left(-E_{n+1}\right)}, & \text { if } \omega_0>0.\end{cases}
 $$
 Note that $E_{n+1}^F$ is the unique vector in the direction $\Phi\left(E_{n+1}\right)$, whose scalar product with $E_{n+1}$ is 1 \cite{Jia-Wang-Xia-Zhang2023}*{eq. (3.2)}. When $\omega_0=0$, one can define it by any unit vector. 

 We refer to \cite{Xia2009-to-appear} for a historical description of  the Hsiung-Minkowski formula.
 In \cite{KKK2016}, Kwong proved the generalization of the Hsiung-Minkowski formula  in a semi-Riemannian manifold with constant curvature.
 In \cite[Theorem 1.1]{He-Li06}, He and Li  generalized the Hsiung-Minkowski formula to anisotropic geometry.
 In \cite{Jia-Wang-Xia-Zhang2023}, Jia, Wang, Xia and Zhang derived an anisotropic Minkowski formula for embedded capillary hypersurfaces in the half space.
 We shall list some details as follows:
\begin{theorem}[Kwong, 2016,  \cite{KKK2016}*{Theorem 1.1}]\label{Thm1.1}
	Suppose $\left(N^{n+1}, g^N\right)$ has constant curvature and $(\Sigma,g)$ is a closed oriented hypersurface. Assume $X \in \Gamma\left(\phi^*(T N)\right)$ is a conformal vector field along $\Sigma$, and $f$ is a smooth function on $\Sigma$. Then for $0 \leq k \leq n-1$, we have
\begin{align}\label{equ:Thm1.1}
	\int_{\Sigma} \alpha f H_k \mathrm{~d}\mu_g=\int_{\Sigma} f H_{k+1} \langle X, \nu\rangle\mathrm{~d}\mu_g-\frac{1}{(n-k)\binom{n}{k}} \int_{\Sigma}\left\langle T_k(\nabla f), X^{\top}\right\rangle \mathrm{~d}\mu_g.
\end{align}
Here $H_k=\sigma_k(\kappa)/\binom{n}{k}$ is the normalized $k$-th mean curvature, Lie derivative of $g^N$ satisfies  $\mathcal{L}_X g^N=2\alpha g^N$ for some function $\alpha$, 
$\nu $ is a unit normal vector field,  ${X}^{\top}$ is the tangential component of $X$ onto $T \Sigma$, $\phi $ is an isometric immersion of $\Sigma$ into $N^{n+1}$, and $T_k$ is the $k$-th Newton transformation of $\mathrm{d}\nu$.
\end{theorem}
\begin{theorem}[Jia-Wang-Xia-Zhang, 2023, \cite{Jia-Wang-Xia-Zhang2023}*{Theorem 1.3}] \label{Thm1.2}
	Assume  $(\Sigma,g) \subset \overline{\mathbb{R}_{+}^{n+1}}$ is a $C^2$ compact anisotropic $\omega_0$-capillary hypersurface, where $\omega_0 \in\left(-F\left(E_{n+1}\right)\right.,$ $\left.F\left(-E_{n+1}\right)\right)$. Let $H_{k+1}^F$ be the normalized anisotropic $(k+1)$-th mean curvature for some $0 \leq k \leq n-1$ and $H_0^F \equiv 1$ by convention. Then it holds
\begin{align}\label{equ:Thm1.2}
	\int_{\Sigma} H_{k}^F\left(F(\nu)+\omega_0\left\langle \nu, E_{n+1}^F\right\rangle\right)-H_{k+1}^F\langle X, \nu\rangle \mathrm{~d}\mu_g=0.
\end{align}

\end{theorem}

In this paper, we first extend Theorem \ref{Thm1.1} and \ref{Thm1.2}, and establish a generalization of the anisotropic  Hsiung-Minkowski  formula as follows:

\begin{theorem}\label{lemma:hminkow-w0}
	Let $\omega_0 \in\left(-F\left(E_{n+1}\right), F\left(-E_{n+1}\right)\right)$ and  $(\Sigma,g,\nabla) \subset \overline{\mathbb{R}_{+}^{n+1}}$ be a $C^2$ compact anisotropic $\omega_0$-capillary hypersurface. Then for any smooth function $f$ on $\Sigma$, there holds
	\begin{align}\label{equ:Thm1.3}
	&	\int_{\Sigma} f H_{k}^F\left(F(\nu)+\omega_0\left\langle \nu, E_{n+1}^F\right\rangle\right)-fH_{k+1}^F\langle X, \nu\rangle \mathrm{~d}\mu_g\nonumber
	\\
	=&-\frac{1}{(n-k)\binom{n}{k}}\int_{\Sigma}\< \nabla f, P_k(F(\nu)\hat{X^{\top}}+\omega_0\<E_{n+1}^F,\nu\>X^{\top }-\omega_0\<X,\nu\>(E_{n+1}^F)^{\top})\> \mathrm{~d}\mu_g,
	\end{align}
	where $k=0,\cdots,n-1,$ $P_k$ is defined in \eqref{equ:P_k}, $H^F_k$ is the anisotropic $k$-th mean curvature,  $F(\nu)\hat{X^{\top}}=F(\nu)X-\<X,\nu\> \nu_F$, $\nu_F$ is the anisotropic normal (see Section \ref{sec 2}), $X^{\top}=X-\<X,\nu\>\nu$ and  $(E^{F}_{n+1})^{\top}=E^{F}_{n+1}-\<E^{F}_{n+1},\nu \>\nu$.
\end{theorem}
\begin{remark}
	(1) If $\partial\Sigma=\emptyset$, i.e., $\Sigma$ is a hypersurface without boundary, then by divergence theorem, we have
	$\int_{\Sigma}\operatorname{div~}(\omega_0fP_k(\<E_{n+1}^F,\nu\>X^{\top}-\<X,\nu\>(E_{n+1}^F)^{\top }))\mathrm{~d}\mu_g=0$, which implies that the following  $$\int_{\Sigma}
	\omega_0 f(n-k)\binom{n}{k}H_k^F\<\nu,E_{n+1}^F\> + \< \nabla f, P_k(\omega_0\<E_{n+1}^F,\nu\>X^{\top }-\omega_0\<X,\nu\>(E_{n+1}^F)^{\top})\>
	\mathrm{~d}\mu_g$$
	vanishes identically by \eqref{equ:pf(2)}.
	Therefore, Theorem \ref{lemma:hminkow-w0} is then the following Hsiung-Minkowski  formula for closed anisotropic hypersurface \cite{Gao-Li2024}
$$
	\int_{\Sigma} f H^F_k F(\nu)\mathrm{~d}\mu_g=\int_{\Sigma} f H^F_{k+1} \left\langle X, \nu\right\rangle\mathrm{~d}\mu_g-\frac{1}{(n-k)\binom{n}{k}} \int_{\Sigma}\left\langle \nabla f, P_k(F(\nu)\hat{X^{\top}})\right\rangle \mathrm{~d}\mu_g .
$$

  (2) Let $\partial\Sigma=\emptyset$. If we take $\W=\mathbb{S}^{n}$, then the above Hsiung-Minkowski  formula matches the formula \eqref{equ:Thm1.1} in Theorem \ref{Thm1.1} with $N^{n+1}=\mathbb{R}^{n+1}$.
	If $f$  is a non-zero constant, then the above Hsiung-Minkowski  formula is just the classical anisotropic Minkowski formula in \cite[Theorem 1.1]{He-Li06}.

  (3) If $\partial\Sigma\neq\emptyset$ and $f$  is a non-zero constant, then \eqref{equ:Thm1.3} in Theorem \ref{lemma:hminkow-w0} matches the formula \eqref{equ:Thm1.2} in Theorem \ref{Thm1.2}.
\end{remark}

 For anisotropic $\omega_0$-capillary hypersurface $\Sigma$, we call the function defined by \begin{align}\label{equ:u}
 	\bar{u}(X)=\frac{\<X,\nu(X)\>}{F(\nu(X))+\omega_0\<\nu(X),E_{n+1}^F\>},  \quad X\in\Sigma,
 \end{align}
 the anisotropic capillary support function of anisotropic capillary hypersurface $\Sigma$ in $\overline{\mathbb{R}_{+}^{n+1}}.$
 Lemma \ref{lemma:u=c} tells us that $\bar{u}\equiv const (\neq 0)$ if and only if $\Sigma$ is an $\omega_0$-capillary Wulff shape.
Assume $\Sigma$ is strictly convex, and $f(X)=f(\bar{u}(X))$,
Theorem \ref{lemma:hminkow-w0} yields the following inequalities, which play a crucial role in proving the  rigidity results.
\begin{corollary}\label{cor:minkowski neq}
		Let $\omega_0 \in\left(-F\left(E_{n+1}\right), F\left(-E_{n+1}\right)\right)$ and  $(\Sigma,g,\nabla) \subset \overline{\mathbb{R}_{+}^{n+1}}$ be a $C^2$, compact, strictly convex, and anisotropic $\omega_0$-capillary hypersurface. Then for $k=0,\cdots,n-1$, and $f$ a smooth function on $\mathbb{R}$, we have
	\begin{align}
		&	\int_{\Sigma} f(\bar{u}) H_{k}^F\left(F(\nu)+\omega_0\left\langle \nu, E_{n+1}^F\right\rangle\right)
		\mathrm{~d}\mu_g
		\leq
		\int_{\Sigma} f(\bar{ u})H_{k+1}^F\langle X, \nu\rangle \mathrm{~d}\mu_g , \text{\ when\ } f'(\bar{u})\geq 0,\label{equ:H-minkowski-w0}
		\\
		&	\int_{\Sigma} f(\bar{u}) H_{k}^F\left(F(\nu)+\omega_0\left\langle \nu, E_{n+1}^F\right\rangle\right)
		\mathrm{~d}\mu_g
		\geq
		\int_{\Sigma} f(\bar{ u})H_{k+1}^F\langle X, \nu\rangle \mathrm{~d}\mu_g , \text{\ when\ } f'(\bar{u})\leq 0.\label{equ:H-minkowski2-w0}
	\end{align}
	Equalities hold if and only if $f(\bar{u})\equiv const$ or $\Sigma$ is an $\omega_0$-capillary Wulff shape.
\end{corollary}



In 1956, Alexandrov \cite{Alexandrov56} proved the classical Alexandrov  theorem that the only compact embedded hypersurface in $\mathbb{R}^{n+1}$ with constant mean curvature is a sphere. In 2010, Koiso and Palmer \cite{Koiso2010} generalized this result and proved that, a smooth immersion of a closed genus zero surface with constant anisotropic mean curvature should be a rescaled Wulff shape.
In 2013, Ma and Xiong \cite{Ma-Xiong2013} gave another proof by applying the evolution  method, to prove that  a closed oriented hypersurface embedded in the Euclidean space $\mathbb{R}^{n+1}$ with constant anisotropic $k$-mean curvature $H^F_k$ ($1\leq k\leq n$) should be a rescaled Wulff shape.
In 2000, Koh \cite{sung2000} derived that a closed oriented hypersurface in Euclidean space $\mathbb{R}^{n+1}$ (or hyperbolic space $\mathbb{H}^{n+1}$ or half sphere $\mathbb{S}_+^{n+1}$) with constant ratio $\frac{H_k}{H_l}$ ($0\leq l< k\leq n$) should be a geodesic hypersphere. In 2006, He and Li \cite{He-Li06} derived the anisotropic Minkowski identity to  prove that a compact convex  hypersurface with constant ratio $\frac{H^F_k}{H^F_l}$ ($0\leq l< k\leq n$) should be a rescaled Wulff shape. In \cite[Corollary 4.10(2)]{KKK2016}, Kwong considered the case of closed oriented hypersurface with constant  ratio $\frac{f(u) H_k}{H_l}$, where $0\leq l< k\leq n$ and $u$ is the support function. In \cite{Mei-Wang-Weng}*{Proposition 3.12}, Mei, Wang and Weng proved that if a capillary hypersurface satisfies $\bar{ u}H_1=1$, where $\bar{ u}$ is the capillary support function, then it must be a spherical cap.




More generally, there are various other rigidity theorems under other conditions, e.g., cases involving the linear combination of distinct higher-order mean curvatures \cite{Stong1960}. Wu and Xia \cite{Wu-Xia2014} obtained  the rigidity problem for hypersurfaces with constant linear combinations
of higher-order mean curvatures in the warped product manifolds. In \cite{Onat2010}, Onat considered the anisotropic case for convex hypersurface. In \cite{Li-Peng}, the second author and Peng weakened the convexity condition, and considered the linear combination of distinct higher-order anisotropic mean curvatures.

Motivated by the previous work, we shall use the weighted anisotropic integral formula of Hsiung-Minkowski
type to deduce the following anisotropic Alexandrov-type results.
\begin{theorem} \label{Thm-Alexan-a=b-w0}
	(1) Let $\omega_0 \in\left(-F\left(E_{n+1}\right), F\left(-E_{n+1}\right)\right)$, $\Sigma \subset \overline{\mathbb{R}_{+}^{n+1}}$ be a $C^2$, compact, strictly convex, and anisotropic $\omega_0$-capillary hypersurface, and $2\leq l\leq r\leq n$. If there are nonnegative and not all vanishing functions $\{a_j(t)\}_{j=l}^{r}$ with $a'_j(t)\geq 0$, nonnegative and not all vanishing functions $\{b_i(t)\}_{i=1}^{l-1}$ with  $b'_i(t)\leq 0$, such that
	\begin{align}
		\sum_{j=l}^{r}a_j(\bar{u})H_j^F=\sum_{i=1}^{l-1}b_i(\bar{u})H_i^F,  \label{equ:thm-aH=b1H-w0}
	\end{align}
	then $\Sigma$ is an  $\omega_0$-capillary Wulff shape.
	
	(2) Let $\omega_0 \in\left(-F\left(E_{n+1}\right), F\left(-E_{n+1}\right)\right)$, $\Sigma \subset \overline{\mathbb{R}_{+}^{n+1}}$ be a $C^2$, compact, embedded, strictly convex, and anisotropic $\omega_0$-capillary hypersurface,
    and $1\leq l\leq r\leq n$. If there are nonnegative and not all vanishing functions $\{a_j(t)\}_{j=l}^{r}$ with $a'_j(t)\geq 0$, nonnegative and not all vanishing functions  $\{b_i(t)\}_{i=0}^{l-1}$, with  $b'_i(t)\leq 0~(1\leq i\leq l-1)$ and $b'_0(t)=0$, such that
	\begin{align}
		\sum_{j=l}^{r}a_j(\bar{u})H_j^F=\sum_{i=0}^{l-1}b_i(\bar{u})H_i^F, \label{equ:thm-aH=b0H-w0}
	\end{align}
	then $\Sigma$ is an  $\omega_0$-capillary Wulff shape.
\end{theorem}
We emphasize that $\Sigma$  in Theorem \ref{Thm-Alexan-a=b-w0}(2) should be embedded, since the use of Heintze-Karcher inequality (i.e., Lemma \ref{lem:H^{-1}-int-fornula-w0}). Here $b_0$ is a constant, which means that Theorem \ref{Thm-Alexan-a=b-w0}(2) contains the case $\sum_{j=1}^{r}a_j(\bar{u})H_j^F=const$.

Specifically, for fixed $k=1,\cdots,n$, taking $l=r=n, a_n\neq0, b_{n-k}\neq 0$ and $ b_i=0 (\forall i\neq {n-k})$ in Theorem \ref{Thm-Alexan-a=b-w0}, we then directly obtain the uniqueness of the solution to the anisotropic capillary Orlicz-Christoffel-Minkowski problem.
\begin{corollary}\label{cor:Christoffel-w0}
	Given a constant $\omega_0 \in\left(-F\left(E_{n+1}\right), F\left(-E_{n+1}\right)\right)$ and a smooth function $f>0$, satisfying $ f^{\prime} \geq 0$, then there exists a unique strictly convex anisotropic $\omega_0$-capillary  hypersurface  $\Sigma\subset \overline{\mathbb{R}^{n+1}_+}$ satisfying the equation   $\sigma_k(\lambda^F)=Cf(\bar{ u})$, and the solution  $\Sigma$ is an $\omega_0$-capillary Wulff shape. Here $C$ is a positive constant, $0<k\leq n$ and $\lambda^F=(\lambda_1^F,\cdots ,\lambda_n^F)=(\frac{1}{\kappa_1^F},\cdots ,\frac{1}{\kappa_n^F})$ are anisotropic  principal curvature radii of $\Sigma$.
\end{corollary}

In \cite{Xia13}, Xia studied the anisotroic Minkowski problem and  proved the existence and uniqueness of the admissible solution to the Monge-Amp\`{e}re type equation on the anisotropic support function.  In \cite{Xia-arxiv}, Mei, Wang, Weng and Xia established a theory for $\theta$-capillary convex bodies in the half-space. And they proved that a capillary convex function (must satisfies the Robin-type boundary condition) yields a capillary convex body (by capillary support function) \cite{Xia-arxiv}*{Proposition 2.6}, and vice versa \cite{Xia-arxiv}*{Lemma 2.4 (3)}. In \cite{Mei-Wang-Weng-2025,Mei-Wang-Weng-Lp-Minkowski}, Mei, Wang, Weng introduced the capillary Minkowski problem and the capillary $L_p$-Minkowski problem with $p>1,\theta\in(0,\frac{\pi}{2}]$ for capillary convex bodies in half-space. They  proved the existence and uniqueness of the admissible solution to the Monge-Amp\`{e}re type  with a Robin boundary value condition. 
 By Corollary \ref{cor:Christoffel-w0}, we use another method to obtain the uniqueness of the capillary $L_p$-Minkowski problem with $p\geq 1$ in the Euclidean capillary convex bodies geometry as follows, and the detailed proof will be presented in Section \ref{sec 4}.

\begin{corollary}
	\label{cor:isotropic-Minkowski}
	Given constants $p\geq1,\theta\in\left(0,\pi\right)$  and a smooth function $\varphi>0$ defined on the sphere cap $\mathbb{S}^n_{\theta}=\{z\in\mathbb{S}^n|\<z,E_{n+1}\>\geq\cos\theta\}$. If there exists a  strictly convex  $\theta$-capillary  hypersurface  $\Sigma\subset \overline{\mathbb{R}^{n+1}_+}$ with support function $u(z)\in C^{\infty}(\mathbb{S}^n_{\theta})$ satisfying
	the equation
	\begin{align}
	\label{equ:Minkowski-Euic}
	u^{1-p}\sigma_n(\lambda)=\varphi,
	\end{align}
	 on $\mathbb{S}_{\theta}^n,$ here  $\lambda=(\lambda_1,\cdots ,\lambda_n)=(\frac{1}{\kappa_1},\cdots ,\frac{1}{\kappa_n})$ are  principal curvature radii of $\Sigma$,
	 then
	
	 (i) for $p\neq n+1$, solution of above  equation is unique;
	
	 (ii) for $p=n+1$, solution of above  equation is unique, up to a rescaling.
\end{corollary}
\begin{remark}
	(1) We remark that the equation \eqref{equ:Minkowski-Euic} is equivalent to the Monge-Amp\`{e}re type equation   
$$u^{1-p}\det\left(u_{ij}+u\delta_{ij}\right)=\varphi,\quad $$
on $\mathbb{S}_{\theta}^n,$ with $(u_{ij}+u\delta_{ij})>0$ on $\mathbb{S}_{\theta}^n$ and $\nabla^{\mathbb{S}}_{\mu}u=\cot\theta u$ along $\partial\mathbb{S}^n_{\theta}$ (see  \cite{Xia-arxiv}),
where $u_{ij}=\nabla^{\mathbb{S}}_i\nabla^{\mathbb{S}}_j u$ is the covariant differentiation of $u$ with respect to an orthonormal frame on $\mathbb{S}^n_{\theta}$, and $\mu$ is the outward  unit co-normal to $\partial\mathbb{S}_{\theta}^n$ on $\mathbb{S}^n$.

(2) Generally, for hypersurface with variable (not constant) tangent angle $\theta$, the uniqueness of hypersurface satisfying $\nu(\Sigma)=\mathcal{A}\subset\mathbb{S}^n$ and equation \eqref{equ:Minkowski-Euic} can be similarly considered.
\end{remark}

In \cite{KKK2018}*{Theorem 1}, Kwong, Lee and Pyo used the weighted Hsiung-Minkowski formulas to prove an Alexandrov-type result for closed embedded hypersurfaces with equation
  $\sum_{j=1}^k\left(b_j(r) H_j+c_j(r) H_1 H_{j-1}\right)=\eta(r) $ in large class of Riemannian warped product manifolds, where $r$ is  radial distance and $\eta(r)$ is  radially symmetric function. We prove a similar conclusion for the  anisotropic $\omega_0$-capillary hypersurface as follows:

 \begin{theorem}\label{thm1.9}
 	Let $\omega_0 \in\left(-F\left(E_{n+1}\right), F\left(-E_{n+1}\right)\right)$, $k=1,\cdots,n,$ and $\Sigma \subset \overline{\mathbb{R}_{+}^{n+1}}$ be a $C^2$, compact, strictly convex, embedded, and anisotropic $\omega_0$-capillary hypersurface.
 	 Let $\left\{b_j(\bar{ u})\right\}_{j=1}^k$ and $\left\{c_j(\bar{ u})\right\}_{j=1}^k$ be two families of monotone increasing, smooth, non-negative, and not all vanishing functions. Suppose
 \begin{align}\label{equ:3.4}
 	\sum_{j=1}^k\left(b_j(\bar{ u}) H^F_j+c_j(\bar{ u}) H^F_1 H^F_{j-1}\right)=\eta(\bar{ u}),
 \end{align}
 for some smooth positive function $\eta(\bar{ u})$ which is monotone decreasing in $\bar{ u}$. Then  $\Sigma$ is an  $\omega_0$-capillary Wulff shape.
 \end{theorem}
Theorem \ref{thm1.9}  contains two special cases worth mentioning: (i) $H^F_k=\eta(\bar{ u})$ and (ii) $H^F_1 H^F_{k-1}=\eta(\bar{ u})$, where $\eta(\bar{ u})$ is a monotone decreasing function. The second case can be regarded as a "non-linear" version of the Alexandrov theorem and seems to be a new phenomenon.

Based on the Minkowski integral formula, Feeman and Hsiung \cite{FH59} proved that if there exists an integer $1\leq s\leq n-1$, with $H_s>0$ and support function $u$ satisfies either $u\leq H_{s-1}/H_s$ or $u\geq H_{s-1}/H_s$, then the hypersurface is a sphere. Later, Onat \cite{Onat2010} generalized it to anisotropic case. In \cite{Stong1960}, Stong proved the rigidity results under the condition $H_{s-1}^{1/{s-1}}\geq c\geq H_{s}^{1/{s}}$ or $H_{s-1}/H_s\geq c\geq H_{s-2}/H_{s-1}$ with constant $c$. The second author and Peng \cite{Li-Peng} generalized them to anisotropic cases. We shall consider the anisotropic capillary hypersurface, and  generalize the constant $c$ to the  function $c(\bar{ u})$.

\begin{theorem}\label{Thm-Alexan-H>c>H-w0}
	Let $\omega_0 \in\left(-F\left(E_{n+1}\right), F\left(-E_{n+1}\right)\right)$, and $\Sigma \subset \overline{\mathbb{R}_{+}^{n+1}}$ be a $C^2$, compact, strictly convex, and anisotropic $\omega_0$-capillary hypersurface.
	
	(1) Suppose there exists an integer $s$, $1<s\leq n$, and a positive function $c(t)$ with $c'(t)\leq 0$, and
	\begin{align}
		\label{equ:thm-H>c>H-w0}
			(H^F_{s-1})^{\frac{1}{s-1}}\geq c(\bar{u})\geq (H^F_{s})^{\frac{1}{s}},
	\end{align}
	at all points of $\Sigma$. Then  $\Sigma$ is an  $\omega_0$-capillary Wulff shape.
	
	(2) Suppose there exists an integer $s$, $1<s\leq n$, and a positive function $c(t)$ with $c'(t)\geq 0$, and
	\begin{align}
		\label{equ:thm-H/H>c>H/H-w0}
	\frac{H^F_{s-1}}{H^F_{s}}\geq c(\bar{u})\geq \frac{H^F_{s-2}}{H^F_{s-1}},
	\end{align}
	at all points of $\Sigma$. Then  $\Sigma$ is an  $\omega_0$-capillary Wulff shape.
\end{theorem}


Similar to  \cite{KKK2018}*{Theorem 3}, the weighted anisotropic Hsiung-Minkowski formula can also be applied to prove rigidity theorem for self-expanding solitons to the so-called weighted generalized inverse anisotropic curvature flow
$$\frac{d}{d t} {X}=\sum_{0 \leq i<j \leq n} a_{i, j}\left(\frac{H^F_i}{H^F_j}\right)^{\frac{1}{j-i}} \nu_F,
$$
 where the weight functions $\left\{a_{i, j}(x) \mid 0 \leq i<j \leq n\right\}$ are non-negative functions on the hypersurface satisfying $\sum_{0 \leq i<j \leq n} a_{i, j}(x)=1$.


\begin{theorem}\label{thm1.10}
	Let $\omega_0 \in\left(-F\left(E_{n+1}\right), F\left(-E_{n+1}\right)\right)$, $\left\{a_{i, j}(x) \mid 0 \leq i<j \leq n\right\}$ be non-negative functions on the hypersurface satisfying $\sum_{0 \leq i<j \leq n} a_{i, j}(x)=1$, $k=\max\{j|a_{i,j}>0 \text{\ for some \ } 0\leq i<j\leq n\}$, and $\Sigma \subset \overline{\mathbb{R}_{+}^{n+1}}$ be a $C^2$, compact, strictly anisotropic   $k$-convex (i.e., $H^F_k>0$), and anisotropic $\omega_0$-capillary hypersurface.
	If there exists a constant $\beta>0$ satisfying
	\begin{align}\label{equ:kkk4.2}
		\sum_{0 \leq i<j \leq k} a_{i, j}\left(\frac{H^F_i}{H^F_j}\right)^{\frac{1}{j-i}}=\beta \bar{ u},
	\end{align}
	 then  $\Sigma$ is an  $\omega_0$-capillary Wulff shape.
\end{theorem}

The rest of this paper is organized as follows. In Section \ref{sec 2}, we briefly introduce some preliminaries on the anisotropic capillary hypersurface. In Section \ref{sec 3},  we show Theorem \ref{lemma:hminkow-w0}, i.e., the generalization of the Hsiung-Minkowski formula, and derived Corollary \ref{cor:minkowski neq}.
In Section \ref{sec 4}, we use the generalization of the Hsiung-Minkowski formula to prove Theorem \ref{Thm-Alexan-a=b-w0} and then the uniqueness of the solution to capillary $L_p$-Minkowski problem with $p\geq 1$. The anisotropic Alexandrov-type Theorems \ref{thm1.9}$\sim$\ref{thm1.10} will be proved in Section \ref{sec 5}.

\section{Preliminary}\label{sec 2}
Let $F$ be a smooth positive function on the standard sphere $(\mathbb{S}^n, g_{\mathbb{S}^n} ,\nabla^{\mathbb{S}})$ such that the matrix
\begin{equation*}
A_{F}(x)~=~\nabla^{\mathbb{S}}\nabla^{\mathbb{S}} {F}(x)+F(x)g_{\mathbb{S}^n},\quad x\in \mathbb{S}^n,
\end{equation*}
is positive definite on $  \mathbb{S}^n$, where 
 $g_{\mathbb{S}^n}$ denotes the round metric on $\mathbb{S}^n$. Then there exists a unique smooth strictly convex hypersurface $\W$ defined by
\begin{align*}
\W=\{\Phi(x)|\Phi(x):=F(x)x+\nabla^{\mathbb{S}} {F}(x),~x\in \mathbb{S}^n\},
\end{align*}
whose support function is given by $F$. We call $\W$ the Wulff shape determined by the function $F\in C^{\infty}(\mathbb{S}^n)$. When $F$ is a constant, the Wulff shape is just a round sphere.

The smooth function $F$ on $\mathbb{S}^n$ can be extended homogeneously to a $1$-homogeneous function on $\mathbb{R}^{n+1}$ by
\begin{equation*}
F(x)=|x|F({x}/{|x|}), \quad x\in \mathbb{R}^{n+1}\setminus\{0\},
\end{equation*}
and setting $F(0)=0$. Then it is easy to show that $\Phi(x)=DF(x)$ for $x\in \mathbb{S}^n$, where $D$ denotes the standard gradient on $\mathbb{R}^{n+1}$. The homogeneous extension $F$ defines a Minkowski norm on $\mathbb{R}^{n+1}$, that is, $F$ is a norm on $\mathbb{R}^{n+1}$ and $D^2(F^2)$ is uniformly positive definite on $\mathbb{R}^{n+1}\setminus\{0\}$. We can define a dual Minkowski norm $F^0$ on $\mathbb{R}^{n+1}$ by
\begin{align*}
F^0(\xi):=\sup_{x\neq 0}\frac{\langle x,\xi\rangle}{{F}(x)},\quad \xi\in \mathbb{R}^{n+1}.
\end{align*}
We call $\W$ the unit Wulff shape since
 $$\W=\{x\in \R^{n+1}: F^0(x)=1\}.$$
A Wulff shape of radius $r_0$ centered at $x_0$ is given by
\begin{align*}
	\W_{r_0}(x_0)=\{x\in\mathbb{R}^{n+1}:F^0(x-x_0)=r_0\}.
\end{align*}
An $\omega_0$-capillary Wulff shape of radius $r_0$ is given by
\begin{align}\label{equ:crwo}
	\W_{r_0,\omega_0}(E):=\{x\in\overline{\mathbb{R}_+^{n+1}}:F^0(x-r_0\omega_0E)=r_0\},
\end{align}
which is a part of a Wulff shape cut by a hyperplane  $\{x_{n+1}=0\}$, here $E$ satisfies $\<E,E_{n+1}\>=1$ which guarantees that  $\W_{r_0,\omega_0}(E)$ satisfies capillary condition \eqref{equ:w0-capillary}. Obviously $E-E_{n+1}^F\in\partial\overline{\mathbb{R}^{n+1}_+}$, since $\<E_{n+1}^F,E_{n+1}\>=0$.

Let $(\Sigma,g)\subset \overline{\mathbb{R}^{n+1}_+}$ be a $C^2$ hypersurface with $\partial\Sigma\subset\partial\mathbb{R}^{n+1}_+$, which encloses a bounded domain $\Omega$. Let $\nu$ be the unit normal of $\Sigma$ pointing outward of $\Omega$.
The anisotropic Gauss map of $\Sigma$  is defined by $$\begin{array}{lll}\nu_F: &&\Sigma\to  \W\\
&&X\mapsto \Phi(\nu(X))=F(\nu(X))\nu(X)+\nabla^\mathbb{S} F(\nu(X)).\end{array} $$

 The anisotropic Weingarten map $S_F$ is the derivative of the anisotropic normal $\nu_F$. The anisotropic principal curvatures $\k^F=(\k^F_1,\cdots, \k^F_n)$ of $\Sigma$ with respect to $\W$ at $X\in \Sigma$  are defined as the eigenvalues of
 \begin{align}\label{equ:S_F}
 	S_F=\mathrm{d}\nu_F=\mathrm{d}(\Phi\circ\nu)=A_F\circ \mathrm{d}\nu : T_X \Sigma\to T_{\nu_F(X)} \W=T_X \Sigma.
 \end{align}
We define the normalized $k$-th elementary symmetric function $H^F_k=\sigma_k(\kappa^F)/\binom{n}{k}$ of the anisotropic principal curvature $\kappa^F$:
\begin{align*}
H^F_k:=\binom{n}{k}^{-1}\sum_{1\leq {i_1}<\cdots<{i_k}\leq n} \kappa^F_{i_1}\cdots \kappa^F_{i_k},\quad k=1,\cdots,n.
\end{align*}
Setting $H^F_{0}=1$ and $H^F_{n+1}=0$
for convenience. To simplify the notation, from now on, we use  $\sigma_k$ to denote function $\sigma_k(\kappa^F)$ of the anisotropic principal curvatures, if there is no confusion.
When $\Sigma=\W\cap\overline{\mathbb{R}^{n+1}_+}\subset\W$, we have $\kappa^F(\Sigma)=(1,\cdots ,1)$ (see \cite{Xia2017}), and then $H^F_k=1$.

A truncated Wulff shape is a part of a Wulff shape cut by a hyperplane  $\{x_{n+1}=0\}$.
Namely, it is an intersection of a Wulff shape and $\mathbb{R}^{n+1}_+$.
It is easy to check that the anisotropic normal of $\W_r(x_0)$ is $\nu_F(x)=\frac{x-x_0}{r}$. On $\W_r(x_0)\cap\{x_{n+1}=0\}$, we have $\<\nu_F,-E_{n+1}\>=\<\frac{x-x_0}{r},-E_{n+1}\>=\<\frac{x_0}{r},E_{n+1}\>$, which is a constant.

The boundary condition \eqref{equ:w0-capillary} implies that $\omega_0 \in(-F(E_{n+1}), F(-E_{n+1}))$ by the Cauchy-Schwarz inequality $\<x,z\>\leq F^0(x)F(z)$ (see \cite{Jia-Wang-Xia-Zhang2023}).
This Cauchy-Schwarz inequality also implies that $\bar{ u}$ in \eqref{equ:u} is well-defined, i.e.
\begin{proposition}[\cite{Jia-Wang-Xia-Zhang2023}*{Prop 3.2}]\label{prop2.1}
	For $\omega_0 \in(-F(E_{n+1}), F(-E_{n+1}))$, it holds that
	\begin{align*}
		F(z)+\omega_0\<z,E_{n+1}^F\>>0,\quad\text{for any}\ z\in \mathbb{S}^n.
	\end{align*}
\end{proposition}

The generalized Newton-Maclaurin inequality \cite{Chen-Guan-Li-Scheuer2022} is as the following.
  \begin{proposition}
  	\label{Prop:Newton-Maclaurin}
  	For $\lambda=(\lambda_1,\cdots , \lambda_n) \in \Gamma_k=\left\{\lambda \in \mathbb{R}^n: \sigma_i(\lambda)>0, \forall 1 \leq i \leq k\right\}$ and $k>l \geq 0, r>s \geq 0, k \geq r, l \geq s$, we have
 \begin{align}
 \label{equ:N-Mneq}
 \left[\frac{\sigma_k(\lambda)/\binom{n}{k} }{\sigma_l(\lambda)/\binom{n}{l} }\right]^{\frac{1}{k-l}} \leq\left[\frac{\sigma_r(\lambda) /\binom{n}{r}}{\sigma_s(\lambda) /\binom{n}{s}}\right]^{\frac{1}{r-s}},
 \end{align}
 with the equality if and only if  $\lambda_1=\cdots=\lambda_n>0$.
  \end{proposition}


 The following important operator $P_k$ introduced in \cite{HL08} is defined by
 \begin{align} \label{equ:P_k}
 	& P_k=\sigma_k I-\sigma_{k-1} S_F+\cdots+(-1)^k S_F^k, \quad k=0, \cdots, n.
 \end{align}
 Obviously, $P_n=0$,
 $
 P_k=\sigma_k I-P_{k-1} S_F,\ k=1,\cdots,n$, and   $\mathrm{d}\nu\circ P_k$ is symmetric for each $k$.
 \begin{lemma}[He-Li, 2008, \cite{HL08}]\label{Lemma2.1} For each $0\leq k\leq n$, and $X\in (\Sigma,g,\nabla)\subset \mathbb{R}^{n+1}$,  we have
	\begin{align}
 		\operatorname{div}\left(P_k\left(\nabla^{\mathbb{S}} F\right) \circ \nu\right)+F\left(\nu\right) \operatorname{tr}\left(P_k \circ \mathrm{d} \nu\right)&=(k+1) \sigma_{k+1},
 		\label{equ:lemma2.1.1}
 	\\
 		\operatorname{div}\left(P_k (X^{\top})\right)+\left\langle X, \nu\right\rangle \operatorname{tr}\left(P_k \circ \mathrm{d} \nu\right)& =(n-k) \sigma_k,
 		\label{equ:lemma2.1.2}
 	\end{align} 	
 and
 $$
  		\sum_{j=1}^{n}(P_k)_{ij,j}=0,\nonumber
 $$
 	where $X^{\top}=X- \left\langle X, \nu\right\rangle \nu$,  
 	 $\operatorname{div}$ is divergence operator respect to metric $g$,  and $(P_k)_{ij,j}=\nabla_{e_j}(P_k)_{ij}$ for an orthonormal frame field $\{e_i\}$ on $\Sigma$.
 	
 \end{lemma}

\section{Proof of Theorem \ref{lemma:hminkow-w0} and Corollary \ref{cor:minkowski neq}}\label{sec 3}

In this section, we denote  $\mu$ the unit outward co-normal of $\partial \Sigma\subset \Sigma$, and denote
\begin{align}
	\xi&=F(\nu)\hat{X^{\top}}+\omega_0\<E_{n+1}^F,\nu\>X^{\top}-\omega_0\<X,\nu\>(E_{n+1}^F)^{\top}\nonumber
	\\
	&=F(\nu)X-\<X,\nu\>\nu_F+\omega_0\<E_{n+1}^F,\nu\>X^{\top}-\omega_0\<X,\nu\>(E_{n+1}^F)^{\top}.\label{equ:xi}
\end{align}

To prove the generalization of anisotropic Hsiung-Minkowski formula, we need the following key lemmas.
\begin{lemma}
	\label{Lem:hiu=0}
	Let $\omega_0 \in\left(-F\left(E_{n+1}\right), F\left(-E_{n+1}\right)\right)$ and $(\Sigma,g,\nabla) \subset \overline{\mathbb{R}_{+}^{n+1}}$ be a $C^2$ compact anisotropic $\omega_0$-capillary hypersurface. Let  $\{e_{\alpha}\}_{\alpha=1}^{n-1}$ be an orthonormal frame of $\partial \Sigma$.
	 Then along $\partial \Sigma$,
	$S_F(e_{\alpha})\in T(\partial \Sigma)$,
	that is
	\begin{align*}
		\<S_F(e_{\alpha}),\mu\>=0, \ \alpha=1,\cdots,n-1  .
	\end{align*}
\end{lemma}
\begin{proof}
	Along $\partial \Sigma$, we have by $\<\nu_F,-E_{n+1}\>=\omega_0$ that
	\begin{align}\label{equ:pfL1-1}
		0=e_{\alpha}(\<\nu_F,-E_{n+1}\>)=\<\mathrm{d} \nu_F(e_{\alpha}),-E_{n+1}\>=\<S_F(e_{\alpha}),-E_{n+1}\>.
	\end{align}
	Since $\mu$ is the unit outward co-normal of $\partial \Sigma\subset \Sigma$, $\{e_1,\cdots,e_{n-1},\mu\}$ is also an orthonormal frame of $\Sigma$. Then we can write $S_F(e_{\alpha})$ as
	\begin{align}\label{equ:pfL1-2}
		S_F(e_{\alpha})=\sum_{\beta=1}^{n-1}\<S_F(e_{\alpha}),e_{\beta}\>e_{\beta}+\<S_F(e_{\alpha}),\mu\>\mu.
	\end{align}
	Since $e_{\alpha}\in T(\partial\Sigma)$ and $\partial\Sigma\subset\partial\mathbb{R}_+^{n+1}$, we have $\<e_{\alpha},E_{n+1}\>=0$. Putting \eqref{equ:pfL1-2} into \eqref{equ:pfL1-1} we obtain for any $\alpha =1,\cdots, n-1$
	\begin{align*}
		\<S_F(e_{\alpha}),\mu\>\<\mu,E_{n+1}\>=0.
	\end{align*}
	
	To finish the proof, we just need to check $\<\mu,E_{n+1}\>\neq 0$. In fact, along $\partial \Sigma$, the three vectors $\nu,\mu,E_{n+1} \in \mathbb{R}^{n+1}$ are all orthogonal to an $(n-1)$-dimensional linear space $T(\partial\Sigma)$, then $\nu,\mu,E_{n+1}$ are on the same plane.
	Since $\nu\bot T \Sigma$ and $\mu\in T \Sigma$, we have $\nu\bot \mu$.
	If $\<\mu,E_{n+1}\>=0$, then $\nu=E_{n+1}$ or $\nu=-E_{n+1}$. Putting it into $\omega_0=\<\Phi(\nu),-E_{n+1}\>$ yields $\omega_0=-F(E_{n+1})$ or $\omega_0=F(-E_{n+1})$, this contradicts to the assumption  $\omega_0 \in\left(-F\left(E_{n+1}\right), F\left(-E_{n+1}\right)\right)$. Therefore $\<\mu,E_{n+1}\>\neq 0$.
\end{proof}

\begin{lemma}\label{lemma3.2}
		Let $\omega_0 \in\left(-F\left(E_{n+1}\right), F\left(-E_{n+1}\right)\right)$ and $(\Sigma,g,\nabla) \subset \overline{\mathbb{R}_{+}^{n+1}}$ be a $C^2$ compact anisotropic $\omega_0$-capillary hypersurface.
	Then along $\partial \Sigma$,
	\begin{align*}
		\<P_k(\xi ),\mu\>=0,\quad k=0,\cdots,n,
	\end{align*}
	where $\xi$ is defined in \eqref{equ:xi}.
\end{lemma}
\begin{proof}
	 Let  $\{e_{\alpha}\}_{\alpha=1}^{n-1}$ be an orthonormal frame of $\partial \Sigma$. Along $\partial\Sigma$, by \eqref{equ:P_k} and Lemma \ref{Lem:hiu=0}, we have
	 \begin{align*}
	 	\<P_k(e_{\alpha}),\mu\>=\sigma_k\<e_{\alpha},\mu\>=0.
	 \end{align*}
	 Since $\xi \in T\Sigma$, we have
	 \begin{align*}
	 	\<P_k(\xi),\mu\>=\<P_k(\<\xi,\mu\>\mu),\mu\>=\<\xi,\mu\>\<P_k(\mu),\mu\>.
	 \end{align*}
	
	 Next we prove the claim that $\<\xi,\mu\>=0$. 
	
	 In the proof of Theorem 1.3 in reference \cite{Jia-Wang-Xia-Zhang2023}, the following equation holds on $\partial\Sigma$:
	 \begin{align*}
	 	-\langle X, \mu\rangle\left\langle \nu, E_{n+1}^F\right\rangle \omega_0+\langle X, \nu\rangle\left\langle\mu, E_{n+1}^F\right\rangle \omega_0
	 	=
	 	F(\nu)\langle X, \mu\rangle-\langle X, \nu\rangle\left\langle \nu_F, \mu\right\rangle.
	 \end{align*}
By $\<\nu,\mu\>=0$, we know $\<X,\mu\>=\<X^{\top},\mu\>$ and $\<E^F_{n+1},\mu\>=\<(E^F_{n+1})^{\top},\mu\>$, thus
\begin{align*}
	0=&\langle X, \mu\rangle\left\langle \nu, E_{n+1}^F\right\rangle \omega_0
	-\langle X, \nu\rangle\left\langle\mu, E_{n+1}^F\right\rangle \omega_0
	+F(\nu)\langle X, \mu\rangle
	-\langle X, \nu\rangle\left\langle \nu_F, \mu\right\rangle
	\\
	=&\<F(\nu)X-\<X,\nu\>\nu_F+\omega_0\<E_{n+1}^F,\nu\>X^{\top}-\omega_0\<X,\nu\>(E_{n+1}^F)^{\top},\mu\>
	\\
	=&\<\xi,\mu\>.
\end{align*}

This completes the proof.
\end{proof}

\begin{lemma}\label{lemma3.3}
	Let $\omega_0 \in\left(-F\left(E_{n+1}\right), F\left(-E_{n+1}\right)\right)$ and $(\Sigma,g,\nabla) \subset \overline{\mathbb{R}_{+}^{n+1}}$ be a $C^2$ compact anisotropic $\omega_0$-capillary hypersurface.
	Then on $ \Sigma$, for $k=0,\cdots, n$ and $\xi$ defined in \eqref{equ:xi}, we have
	\begin{align*}
		\operatorname{div}(P_k(\xi)) =(n-k)\binom{n}{k}
		\left(F(\nu)H^F_k+\omega_0\<E^F_{n+1},\nu\>
		H^F_k-\<X,\nu\>H^F_{k+1}
		\right).
	\end{align*}
\end{lemma}
\begin{proof}
	For $X\in\Sigma$, we denote that $\hat{u}(X)=\frac{\<X,\nu(X)\>}{F(\nu(X))}$ and $\hat{X^{\top}}=X-\hat{ u}\nu_F=X^{\top}-\hat{ u}\nabla^{\mathbb{S}}F$, then we can check that $\nabla\hat{ u}=F(\nu)^{-1}\mathrm{d}\nu(X^{\top})-\<X,\nu\>F(\nu)^{-2}\mathrm{d}\nu(\nabla^{\mathbb{S}}F)=F(\nu)^{-1}\mathrm{d}\nu(\hat{X^{\top}})$.
	
	By \eqref{equ:lemma2.1.1} and \eqref{equ:lemma2.1.2}, we have that
	\begin{align*}
		\operatorname{div}(P_{k}(X^{\top}))-\hat{ u}\operatorname{div}(P_{k}(\nabla^{\mathbb{S}}F\circ \nu))=	{(n-k)\binom{n}{k}}(H^F_{k}-H^F_{k+1}\hat{ u})		
		,\ k=0,\cdots, n.
	\end{align*}
	Combining with
	\begin{align*}
		\operatorname{div}(\hat{ u}P_k(\nabla^{\mathbb{S}}F))=\< \nabla\hat{ u}, P_k(\nabla^{\mathbb{S}}F)\>+\hat{ u}\operatorname{div}(P_k(\nabla^{\mathbb{S}}F)),
	\end{align*}
	we can derive
	\begin{align}
		&\operatorname{div}\left(
		F(\nu)P_k(\hat{X^{\top}})
		\right)\nonumber
		\\
		=&\< \nabla F(\nu), P_k(\hat{X^{\top}})\>
		+F(\nu)\operatorname{div}(P_k(X^{\top}))-F(\nu)\operatorname{div}(\hat{ u}P_k(\nabla^{\mathbb{S}}F))\nonumber
		\\
		=&\< \mathrm{d}\nu(\nabla^{\mathbb{S}}F), P_k(\hat{X^{\top}})\>
		+F(\nu)(n-k)\binom{n}{k}(H^F_k-H^F_{k+1}\hat{ u})-F(\nu)\<\nabla\hat{ u},P_k(\nabla^{\mathbb{S}}F)\>\nonumber
		\\
		=&\<\nabla^{\mathbb{S}}F,\mathrm{d}\nu\circ P_k(\hat{X^{\top}})\>
		+F(\nu)(n-k)\binom{n}{k}(H^F_k-H^F_{k+1}\hat{ u})-\<\hat{X^{\top}},\mathrm{d}\nu\circ P_k(\nabla^{\mathbb{S}}F)\>\nonumber
		\\
		=&F(\nu)(n-k)\binom{n}{k}(H^F_k-H^F_{k+1}\hat{ u}).\label{equ:pfthm1.3-1}
	\end{align}
	where the last equality is due to the symmetry of $\mathrm{d}\nu\circ P_k$.
	
	Using Lemma \ref{Lemma2.1} and Weingarten  formula $D_{e_i}e_j=-h_{ij}\nu$ where $h_{ij}$ denotes the second fundamental form under the normal coordinates $\{e_i\}_{i=1}^n\subset T_X\Sigma$, we have
	\begin{align}
		&\operatorname{div}\left(\<E_{n+1}^F,\nu\>P_k(X^{\top})\right)\nonumber
		\\
		=&\<\nabla(\<E^F_{n+1},\nu\>),P_k(X^{\top})\>+\<E^F_{n+1},\nu\>\operatorname{div}(P_k(X^{\top}))\nonumber
		\\
		=&\<\mathrm{d}\nu(E^F_{n+1})^{\top},P_k(X^{\top})\>+(n-k)\sigma_k\<E^F_{n+1},\nu\>-\<E^F_{n+1},\nu\>\<X,\nu\>\operatorname{tr}(P_k\circ \mathrm{d}\nu),\nonumber
	\end{align}
	and
	\begin{align*}
		&\operatorname{div}\left(
		\<X,\nu\>P_k((E_{n+1}^F)^{\top})
		\right)
		\\
		=&\<\nabla(\<X,\nu\>),P_k((E_{n+1}^F)^{\top})\>
		+\<X,\nu\>\left(
		(P_k)_{ij}\<E^F_{n+1},e_i\>
		\right)_{,j}
		\\
		=&\<\mathrm{d}\nu(X^{\top}),P_k((E_{n+1}^F)^{\top})\>
		+\<X,\nu\>\left(
		(P_k)_{ij}\<E^F_{n+1},-h_{ij}\nu
		\right)
		\\
		=&\<X^{\top},\mathrm{d}\nu\circ P_k((E_{n+1}^F)^{\top})\>-\<E^F_{n+1},\nu\>\<X,\nu\>\operatorname{tr}(P_k\circ \mathrm{d}\nu),
	\end{align*}
	which implies
	\begin{align}\label{equ:pf(2)}
		\operatorname{div}\left(
		P_k(\omega_0(\<E_{n+1}^F,\nu\>X^{\top}-\<X,\nu\>(E^F_{n+1})^{\top})
		\right)=\omega_0
		(n-k)\sigma_k\<E^F_{n+1},\nu\>.
	\end{align}
	Combination of \eqref{equ:pfthm1.3-1} and \eqref{equ:pf(2)} yields
	\begin{align*}
		\operatorname{div}(P_k(\xi)) =F(\nu)(n-k)\binom{n}{k}(H^F_k-H^F_{k+1}\hat{ u})
		+\omega_0
		(n-k)\sigma_k\<E^F_{n+1},\nu\>.
	\end{align*}
\end{proof}

\begin{proof}[\textbf{Proof of Theorem \ref{lemma:hminkow-w0}}]

	Using Lemmas \ref{lemma3.2}, \ref{lemma3.3} and the divergence theorem we have
	\begin{align*}
		0=&\int_{\partial\Sigma}
		\<f\cdot P_k(\xi),\mu\>
		\mathrm{~d}\mu_{\partial\Sigma}
		=\int_{\Sigma}
		\operatorname{div}\left(
		f\cdot P_k(\xi)
		\right)
		\mathrm{~d}\mu_g
		\\
		=&\int_{\Sigma}
		\<\nabla f,P_k(\xi)\>
		+f\operatorname{div}\left(
		P_k(\xi)
		\right)
		\mathrm{~d}\mu_g
		\\
		=&\int_{\Sigma}
		\<\nabla f,P_k(\xi)\>
		+f(n-k)\binom{n}{k}
		\left(F(\nu)H^F_k+\omega_0\<E^F_{n+1},\nu\>
		H^F_k-\<X,\nu\>H^F_{k+1}
		\right)
		\mathrm{~d}\mu_g,
	\end{align*}
	which yields the weighted anisotropic Hsiung-Minkowski formula \eqref{equ:Thm1.3}.
\end{proof}

\begin{lemma}\label{lemma:u=c}
	Let $\omega_0 \in\left(-F\left(E_{n+1}\right), F\left(-E_{n+1}\right)\right)$ and $(\Sigma,g,\nabla) \subset \overline{\mathbb{R}_{+}^{n+1}}$ be a $C^2$, compact,  strictly convex, and anisotropic $\omega_0$-capillary hypersurface.  Then
	$\bar{ u}$ is a non-zero constant, if and only if $\Sigma$ is an $\omega_0$-capillary Wulff shape, which is a part of a Wulff shape $\W_{r_0}(x_0),x_0=r_0\omega_0E_{n+1}^F.$
\end{lemma}
\begin{proof}
	$(\Leftarrow)$ If $\Sigma= \W_{r_0}(x_0)\cap \overline{\mathbb{R}_{+}^{n+1}},\ x_0=r_0\omega_0E_{n+1}^F.$
	It easy to check that $\Sigma$ is an anisotropic  $\omega_0$-capillary hypersurface, since $\<\frac{X-x_0}{r_0},-E_{n+1}\>=\omega_0$ for $X\in\partial\Sigma.$
	
	For $X\in \Sigma$, we have $\nu_F=\frac{X-x_0}{r_0}$, then
	\begin{align*}
		\frac{1}{r_0}\<X,\nu\>=\<\frac{X-x_0}{r_0}+\frac{x_0}{r_0},\nu\>=\<\nu_F,\nu\>+\<\frac{x_0}{r_0},\nu\>=F(\nu)+\<\omega_0E_{n+1}^F,\nu\>>0,
	\end{align*}
	where in the last inequality we use Proposition \ref{prop2.1}. Then $\bar{ u}=\frac{\<X,\nu\>}{F(\nu)+\omega_0\<E_{n+1}^F,\nu\>}\equiv r_0$.\\
	
	$(\Rightarrow)$ If $\bar{ u}\equiv r_0$, then
	\begin{align}
		\label{equ:pfL2-1}
		\<X,\nu(X)\>=r_0\left(
		F(\nu(X))+\omega_0\<\nu(X),E_{n+1}^F\>
		\right),\quad X\in\Sigma.
	\end{align}
	Let $\{e_i\}_{i=1}^{n}$ be an orthonormal frame of $ \Sigma$, and $h_i^j$ be the Weingarten transformation matrix.
	Using  Gauss formula $\partial_i\nu=h_i^ke_k$ and $\nabla F(\nu)=\mathrm{d}\nu(\nabla^{\mathbb{S}}F\circ \nu)$, we have
	\begin{align*}
		&	\<e_i,\nu\>+\<X,D_{e_i}\nu\>=r_0D_{e_i}(F(\nu))+r_0\omega_0\<D_{e_i}\nu,E_{n+1}^F\>
		\\
		\Rightarrow&
		\<X,h^k_ie_k\>=r_0h^k_i\<\nabla^{\mathbb{S}}F\circ\nu,e_k\>+r_0\omega_0\<h^k_ie_k,E_{n+1}^F\>
		\\
		\Rightarrow&h_i^k\<e_k,X-r_0\nabla^{\mathbb{S}}F\circ \nu-r_0\omega_0E_{n+1}^F\>=0.
	\end{align*}
	The strictly convexity of $\Sigma$ implies $\{h_i^k\}>0$, thus $\<e_k,X-r_0\nabla^{\mathbb{S}}F\circ \nu-r_0\omega_0E_{n+1}^F\>=0$, for $k=1,\cdots,n$.
	Then there exists a function $\lambda(X)$ satisfies $X-r_0\nabla^{\mathbb{S}}F\circ \nu-r_0\omega_0E_{n+1}^F=\lambda\nu$, that is
	\begin{align}\label{equ:pfL2-2}
		X=r_0\nabla^{\mathbb{S}}F\circ \nu(X)+r_0\omega_0E_{n+1}^F+\lambda\nu(X).
	\end{align}
	Putting \eqref{equ:pfL2-2} into \eqref{equ:pfL2-1} yields
	\begin{align*}
		\<r_0\nabla^{\mathbb{S}}F\circ \nu+r_0\omega_0E_{n+1}^F+\lambda\nu, \nu\>=r_0\left(
		F(\nu)+\omega_0\<\nu,E_{n+1}^F\>
		\right),
	\end{align*}
	which implies $\lambda=r_0F(\nu)$. Then \eqref{equ:pfL2-2} can be rewritten as
	\begin{align*}
		X=r_0\nabla^{\mathbb{S}}F\circ \nu+r_0\omega_0E_{n+1}^F+r_0F(\nu)\nu=r_0\nu_F+r_0\omega_0E_{n+1}^F.
	\end{align*}
	Assume $x_0=r_0\omega_0E_{n+1}^F$, we have
	\begin{align*}
		F^0(X-x_0)=F^0(r_0\nu_F)=r_0, \quad X\in \Sigma,
	\end{align*}
	i.e., $\Sigma\subset \W_{r_0}(x_0)$. Since  $\Sigma\subset \overline{\mathbb{R}_{+}^{n+1}}$ is a compact,  anisotropic $\omega_0$-capillary hypersurface, we have $\Sigma= \W_{r_0}(x_0)\cap \overline{\mathbb{R}_{+}^{n+1}}.$
\end{proof}
\begin{proof}[\textbf{Proof of Corollary \ref{cor:minkowski neq}}]
	Since
 \begin{align}
 	\nabla\bar{u}
 	=&{(F(\nu)+\omega_0\<\nu,E_{n+1}^F\>)^{-2}}
 	\left(
 	(F(\nu)+\omega_0\<\nu,E_{n+1}^F\>)\nabla(\<X,\nu\>)
 		\right. \nonumber
 		\\
 		&\left.
 		\qquad\qquad\qquad\qquad\qquad\qquad-
 		\<X,\nu\>(\nabla F(\nu)+\omega_0\nabla(\<E_{n+1}^F,\nu\>))\right) \nonumber
 	\\
 	=&(F(\nu)+\omega_0\<\nu,E_{n+1}^F\>)^{-2} \mathrm{d}\nu\left(F(\nu)\hat{X^{\top}}+\omega_0\<E_{n+1}^F,\nu\>X^{\top}-\omega_0\<X,\nu\>(E_{n+1}^F)^{\top}\right) \nonumber
 	\\
 	=&\left(F(\nu)+\omega_0\<E_{n+1}^F,\nu\>\right)^{-2}\mathrm{d}\nu(\xi), \label{equ:du=dv(Xi)}
 \end{align}
	we have
	$$\nabla (f(\bar{ u}))=f'(\bar{ u})\nabla\bar{ u}=f'(\bar{ u})\left(F(\nu)+\omega_0\<E_{n+1}^F,\nu\>\right)^{-2}\mathrm{d}\nu(\xi).$$
	By Theorem \ref{lemma:hminkow-w0} and the  symmetry of $\mathrm{d} \nu$, we obtain
	\begin{align*}
		&\int_{\Sigma} f(\bar{ u})\left(H^F_k(F(\nu)+\omega_0 \<\nu,E_{n+1}^F\>)-\<X,\nu\>H^F_{k+1}\right) \mathrm{~d}\mu_g
		\\
		=&
		-\frac{1}{\binom{n}{k}(n-k)}\int_{\Sigma}\< \nabla( f(\bar{ u})), P_k(\xi)\> \mathrm{~d}\mu_g
		\\
		=&-\frac{1}{\binom{n}{k}(n-k)}\int_{\Sigma}\< f'(\bar{ u})\left(F(\nu)+\omega_0\<E_{n+1}^F,\nu\>\right)^{-2}\mathrm{d}\nu(\xi), P_k(\xi)\> \mathrm{~d}\mu_g
		\\
		=&-\frac{f'(\bar{ u})\left(F(\nu)+\omega_0\<E_{n+1}^F,\nu\>\right)^{-2}}{\binom{n}{k}(n-k)}\int_{\Sigma}\< \xi, \mathrm{d}\nu\circ P_k(\xi)\> \mathrm{~d}\mu_g.
	\end{align*}
	Since $\Sigma$ is strictly convex hypersurface, we know $ \mathrm{d}\nu \circ P_k$ is positive define. If $f'(\bar{ u})\geq 0$ in $\Sigma$, we have
	\begin{align*}
		\int_{\Sigma} f(\bar{ u})\left(H^F_k(F(\nu)+\omega_0 \<\nu,E_{n+1}^F\>)-\<X,\nu\>H^F_{k+1}\right) \mathrm{~d}\mu_g\leq 0,
	\end{align*}
	where	equality holds if and only if $f'(\bar{ u})=0$ or $\xi=0$, the latter is equivalent to $ \bar{ u}\equiv const$ by \eqref{equ:du=dv(Xi)}, that means
	$\Sigma$ is an $\omega_0$-capillary Wulff shape in view of Lemma \ref{lemma:u=c}.
	
	The proof of case $f'(\bar{ u})\leq 0$ is similar.
\end{proof}

\section{Proof of theorem \ref{Thm-Alexan-a=b-w0} and uniqueness of the $L_p$-Minkowski problem}\label{sec 4}

The aim of this section is to derive Theorem \ref{Thm-Alexan-a=b-w0}, and as application, we then provide a new proof of the uniqueness of the solution to $L_p$-Minkowski problem with $p\geq 1$ in the Euclidean capillary convex bodies geometry.

First, we need the following anisotropic Heintze-Karcher inequality.
\begin{lemma}[\cite{Jia-Wang-Xia-Zhang2023}*{Theorem 1.2}] \label{lem:H^{-1}-int-fornula-w0}
	Let $\omega_0 \in\left(-F\left(E_{n+1}\right), F\left(-E_{n+1}\right)\right)$ and $\Sigma \subset \overline{\mathbb{R}_{+}^{n+1}}$ be a $C^2$ compact embedded strictly $F$-mean convex (i.e. $H_1^F>0$) hypersurface with boundary $\partial \Sigma \subset$ $\partial \mathbb{R}_{+}^{n+1}$ such that
	$$
	\left\langle\Phi(\nu(x)),-E_{n+1}\right\rangle=\omega(x) \leq \omega_0, \quad \text { for any } x \in \partial \Sigma.
	$$
	Then it holds that, for $ X\in\Sigma$, we have
	\begin{align}\label{equ:H_k-wo}
		\int_{\Sigma} \frac{F(\nu)+\omega_0\left\langle \nu, E_{n+1}^F\right\rangle}{H_1^F} \mathrm{~d}\mu_g \geq {(n+1)}|\Omega|=\int_{\Sigma}\<X,\nu\> \mathrm{~d}\mu_g,
	\end{align}	
	where equality  holds if and only if $\Sigma$ is an $\omega_0$-capillary Wulff shape.
\end{lemma}
We also have the following algebraic lemma which is an extension of \cite{Li-Peng}*{Lemma 3.2}.
\begin{lemma}\label{Lem:aH>bH}
	For any vector $\kappa^F=(\kappa^F_1,\cdots,\kappa^F_n)\in \Gamma_r$, we still denote $H_k^F=\frac{\sigma_k(\kappa^F)}{\binom{n}{k}}\ (k=1,\cdots,n)$ and $H^F_0=1$. If there exist nonnegative and not all vanishing numbers $\{a_j\}_{j=l}^{r}$, nonnegative and not all vanishing numbers $\{b_i\}_{i=0}^{l-1}$, $1\leq l\leq r\leq n$, such that
	\begin{align}\label{equ:lem4.2}
	\sum_{j=l}^{r}a_jH^F_j=\sum_{i=0}^{l-1}b_iH^F_i,
	\end{align}
	then
		\begin{align}\label{equ:lem-aH>bH}
		\sum_{j=l}^{r}a_jH^F_{j-1}
		\geq
		\frac{b_0}{H^F_1}
		+
		\sum_{i=1}^{l-1}b_iH^F_{i-1},
	\end{align}
	and equality holds if and only if $\kappa^F_1=\cdots=\kappa^F_n$.
\end{lemma}
\begin{proof}
	Since $\lambda\in\Gamma_r\subset\Gamma_{r-1}\subset\cdots\subset\Gamma_1$ and $\{a_j\}_{j=l}^r$ are not all vanishing, we have $$\sum_{j=l}^{r}a_j H^F_{j}>0.$$
	By Newton-Maclaurin inequality (Proposition \ref{Prop:Newton-Maclaurin}), we know
	\begin{align}\label{equ:pfLem-aH>bH-1}
		H^F_i
		H^F_{j-1}
			\geq
	H^F_{i-1}
		H^F_{j}, \quad 1\leq i<j\leq r.
	\end{align}
	Multiplying \eqref{equ:pfLem-aH>bH-1} by $a_j$ and $b_i$ and summing over $i\in \{1,\cdots,l-1\}$ and $j\in \{l,\cdots,r\}$, combining with the assumption \eqref{equ:lem4.2}, we have
	\begin{align}
		&\sum_{i=1}^{l-1}b_iH^F_i\cdot
		\sum_{j=l}^{r}a_jH^F_{j-1}
		\geq
		\sum_{i=1}^{l-1}b_i H^F_{i-1}\cdot
		\sum_{j=l}^{r}a_jH^F_{j}
	\label{equ:pfLem-aH>bH-3}
		\\
		\Leftrightarrow&
		\left(
		\sum_{j=l}^{r}a_jH^F_j
		-b_0
		\right)\cdot
		\sum_{j=l}^{r}a_jH^F_{j-1}
		\geq
		\sum_{i=1}^{l-1}b_iH^F_{i-1}\cdot
		\sum_{j=l}^{r}a_jH^F_{j}
		\nonumber
		\\
		\Leftrightarrow&
		\left(
		\sum_{j=l}^{r}a_jH^F_{j-1}
		-
		\sum_{i=1}^{l-1}b_iH^F_{i-1}
		\right)
		\sum_{j=l}^{r}a_jH^F_j
		\geq b_0\sum_{j=l}^{r}a_jH^F_{j-1}. \label{equ:pfLem-aH>bH-2}
	\end{align}
	
	If $b_0=0$, then \eqref{equ:pfLem-aH>bH-2} implies \eqref{equ:lem-aH>bH}. Obviously, the equality holds in \eqref{equ:lem-aH>bH} if and only if the equality holds in \eqref{equ:pfLem-aH>bH-3}, that is
	 \begin{align*}
	 	\sum_{i=1}^{l-1}\sum_{j=l}^{r}b_ia_j
	 	\left(
	 	H^F_iH^F_{j-1}
	 	-
	 	H^F_{i-1}H^F_{j}\right)=0.
	 \end{align*}
	 Since $\{a_j\}_{j=l}^r$ are not all vanishing, and $\{b_i\}_{j=1}^{l-1}$ are also not all vanishing, we know there exist $\ 1\leq i_0\leq l-1<l\leq j_0\leq r$, such that the equality holds in \eqref{equ:pfLem-aH>bH-1} for $i=i_0,\ j=j_0$. Then $\kappa^F_1=\cdots=\kappa^F_n$.
	
	 If $b_0\neq 0$. Using Newton-Maclaurin inequality (Proposition \ref{Prop:Newton-Maclaurin}) again, we obtain
	 \begin{align}
	 	\label{equ:pfLem-aH>bH-4}
	 	\sum_{j=l}^{r}a_jH^F_{j-1}
	 	\geq
	 	\sum_{j=l}^{r}a_jH^F_{j} \cdot
	 	\frac{1}{H^F_1}.
	 \end{align}
	From \eqref{equ:pfLem-aH>bH-4}  and \eqref{equ:pfLem-aH>bH-2}, we have    \eqref{equ:lem-aH>bH}. And the equality holds in \eqref{equ:lem-aH>bH} if and only if the equality holds in \eqref{equ:pfLem-aH>bH-4}, that is $\kappa^F_1=\cdots=\kappa^F_n$.
\end{proof}

\begin{proof}[\textbf{Proof of Theorem \ref{Thm-Alexan-a=b-w0}}]
	(1) Multiplying equation \eqref{equ:thm-aH=b1H-w0} with $\<X,\nu\>$, integrating over $\Sigma     $, and using Corollary \ref{cor:minkowski neq} and  Lemma \ref{Lem:aH>bH}, we have
	\begin{align*}
		0&=\int_{\Sigma}
		\left(
		\sum_{j=l}^{r}a_j(\bar{ u})H_j^F
		-
		\sum_{i=1}^{l-1}b_i(\bar{ u})H_i^F
		\right)\<X,\nu\>
		\mathrm{~d}\mu_g
		\\
		&\geq
		\int_{\Sigma}
		\left(
		\sum_{j=l}^{r}a_j(\bar{ u})H_{j-1}^F
		-
		\sum_{i=1}^{l-1}b_i(\bar{ u})H_{i-1}^F
		\right)
		\left(F(\nu)+\omega_0 \<\nu,E_{n+1}^F\>\right)
		\mathrm{~d}\mu_g
		\\
		&\overset{\eqref{equ:lem-aH>bH}}{\geq}0.
	\end{align*}
	Then the equality holds in \eqref{equ:H-minkowski-w0},  \eqref{equ:H-minkowski2-w0} and \eqref{equ:lem-aH>bH}, that is $\Sigma$ is an  $\omega_0$-capillary Wulff shape.
	
	(2) 	Multiplying equation \eqref{equ:thm-aH=b0H-w0} with $\<X,\nu\>$, integrating over $\Sigma$, and using Corollary \ref{cor:minkowski neq}, Lemma \ref{Lem:aH>bH} and \ref{lem:H^{-1}-int-fornula-w0}, we have
	\begin{align*}
		0
		&=\int_{\Sigma}
		\left(
		\sum_{j=l}^{r}a_j(\bar{ u})H_j^F
		-
		\sum_{i=1}^{l-1}b_i(\bar{ u})H_i^F
		\right)\<X,\nu\>
		\mathrm{~d}\mu_g
		-\int_{\Sigma}
		b_0\<X,\nu\>
		\mathrm{~d}\mu_g
		\\
		&\geq
		\int_{\Sigma}
		\left(
		\sum_{j=l}^{r}a_j(\bar{ u})H_{j-1}^F
		-
		\sum_{i=1}^{l-1}b_i(\bar{ u})H_{i-1}^F
			\right)
		\left(F(\nu)+\omega_0 \<\nu,E_{n+1}^F\>\right)
		\mathrm{~d}\mu_g
		-\int_{\Sigma}
		b_0\<X,\nu\>
		\mathrm{~d}\mu_g
		\\
		&\overset{\eqref{equ:lem-aH>bH}}{\geq}
		\int_{\Sigma}
		b_0\frac{F(\nu)+\omega_0 \<\nu,E_{n+1}^F\>}{H_1^F}
		\mathrm{~d}\mu_g
		-\int_{\Sigma}
		b_0\<X,\nu\>
		\mathrm{~d}\mu_g
		\\
		&\overset{\eqref{equ:H_k-wo}}{\geq} 0.
	\end{align*}
	Then the equality holds in \eqref{equ:H-minkowski-w0},  \eqref{equ:H-minkowski2-w0}, \eqref{equ:H_k-wo} and \eqref{equ:lem-aH>bH}, therefore $\Sigma$ is an  $\omega_0$-capillary Wulff shape.
\end{proof}

\begin{proof}[\textbf{Proof of Corollary \ref{cor:isotropic-Minkowski}}] Suppose the  strictly convex  $\theta$-capillary  hypersurfaces  $\Sigma_0,\Sigma\subset \overline{\mathbb{R}^{n+1}_+}$ (with support functions $u_0(z),u(z)\in C^{\infty}(\mathbb{S}^n_{\theta})$ and principal curvature radii $\lambda_0,\lambda$) satisfying
	the equation \eqref{equ:Minkowski-Euic}, which means
	\begin{align}
		\label{equ-pf-Minkow-1}
		u^{1-p}\sigma_n(\lambda)=\varphi=u_0^{1-p}\sigma_n(\lambda_0), \quad\text{on}\ \mathbb{S}_{\theta}^n,
	\end{align}
	for $p\geq 1$.
	
	We first prove  that,  there exists a  constant number $r_0\in\mathbb{R}_{+}$
	such that $\Sigma=r_0\Sigma_0$.
	
	Take a Wulff shape $\W$ (with respect to $F$) which satisfies $\W\cap\overline{\mathbb{R}^{n+1}_+}=\Sigma_0$. Then $\Sigma_0$ is an anisotropic $\omega_0$-capillary hypersurface with $\omega_0=0$ with respect to Wulff shape $\W$, since the anisotropic Gauss map of $\W$ is equal to the position vector of $\W$, and  for position vector $X_0\in\partial\Sigma_0\subset\partial\overline{\mathbb{R}^{n+1}_+}$ we have $\<X_0,E_{n+1}\>=0$.
	
	Since  $\Sigma_0$ and $\Sigma$ are  $\theta$-capillary hypersurfaces, the Gauss maps $\nu_0:\Sigma_0\to\mathbb{S}_{\theta}^n$ and $\nu:\Sigma\to\mathbb{S}_{\theta}^n$ are diffeomorphisms (see \cite{Xia-arxiv}*{Lemma 2.2}), which implies that, for any $X\in\Sigma$, there exists a  $X_0\in\Sigma_0$ satisfies $\nu_0(X_0)=\nu(X)$. Thus, we have $$\<\Phi(\nu(X)),E_{n+1}\>=\<\Phi(\nu_0(X_0)),E_{n+1}\>=0,$$
	since $\Sigma_0$ is an anisotropic $\omega_0$-capillary hypersurface with $\omega_0=0$. This means that $\Sigma$ is also an anisotropic $\omega_0$-capillary hypersurface with $\omega_0=0$ with respect to Wulff shape $\W$.
	
	On the other hands, we can see that $F$ is the support function of $\W$, then $u_0(z)=F(z)$. Combining with $\omega_0=0$ and \eqref{equ:u}, we have that the anisotropic capillary support function of $X\in\Sigma$ is
	\begin{align}
		\label{equ-pf-Minkow-2}
		\bar{u}(X)=\bar{u}(\nu^{-1}(z))=\frac{u(z)}{u_0(z)},\quad z=\nu(X)\in\mathbb{S}_{\theta}^n.
	\end{align}
	Since \eqref{equ:S_F} and the fact that the eigenvalues of $A_F$ are the principal curvature radii of $\W$ (see \cite{Xia2017}), the anisotropic Gauss-Kronecker curvature $H_n^F$ of $\Sigma$ is
	\begin{align}
		\label{equ-pf-Minkow-3}
		H_n^F=\frac{1}{\sigma_n(\lambda^F)}=\frac{\sigma_n(\lambda_0)}{\sigma_n(\lambda)}.
	\end{align}
	Putting \eqref{equ-pf-Minkow-2} and \eqref{equ-pf-Minkow-3} into \eqref{equ-pf-Minkow-1}, it follows that the $\omega_0$-capillary hypersurface $\Sigma$ satisfies
	\begin{align*}
		\sigma_n(\lambda^F)=f(\bar{u}),
	\end{align*}
	where $f(\bar{u})(:=\bar{u}^{p-1})$, is a positive function (for $\bar{u}>0$) satisfying $f'\geq0$ when $p\geq 1$. By Corollary \ref{cor:Christoffel-w0}, $\Sigma$ is an $\omega_0$-capillary Wulff shape, that means there exist a constant number  $r_0\in\mathbb{R}_+$ and  a constant vector $E$ with $\<E,E_{n+1}\>=1$ satisfying $\Sigma=\W_{r_0,\omega_0}(E)=\{x\in\overline{\mathbb{R}^{n+1}_+}:F^0(x-r_0\omega_0E)=r_0\}$.
	Since $\omega_0=0$ and
	$\Sigma_0=\{x\in\overline{\mathbb{R}^{n+1}_+}:F^0(x)=1\}$, we have $$\Sigma=r_0\Sigma_0.$$
	
	We can check that $u=r_0u_0,\lambda=r_0\lambda_0$, putting them into \eqref{equ-pf-Minkow-1} we have
	\begin{align}
		\varphi=u^{1-p}\sigma_n(\lambda)=r_0^{n+1-p}\varphi, \quad\text{on}\ \mathbb{S}_{\theta}^n.
	\end{align}
	If $p\neq n+1$, we have $r_0=1$, since $\varphi>0$.
	
	Then we complete the proof.
\end{proof}

\section{Proof of Alexandrov-type theorems \ref{thm1.9}$\sim$\ref{thm1.10}} \label{sec 5}

This section proves the anisotropic Alexandrov-type theorems \ref{thm1.9}$\sim$\ref{thm1.10} for anisotropic capillary hypersurfaces in the half-space.

\begin{proof}[\textbf{Proof of Theorem \ref{thm1.9}}]
	By dividing \eqref{equ:3.4} by $\eta(\bar{ u}),$ it suffices to prove the result in the case where
	\begin{align}\label{equ:3.5}
		\sum_{j=1}^k\left(b_j(\bar{ u}) H^F_j+c_j(\bar{ u}) H^F_1 H^F_{j-1}\right)=1.
	\end{align}
	Assume $j=1,\cdots, k$. It follows from  \eqref{equ:H-minkowski-w0} that
	\begin{align}\label{equ:kkk3.6}
		\int_{\Sigma} b_j(\bar{u}) H_{j-1}^F\left(F(\nu)+\omega_0\left\langle \nu, E_{n+1}^F\right\rangle\right)
		\mathrm{~d}\mu_g
		\leq
		\int_{\Sigma} b_{j}(\bar{ u})H_{j}^F\langle X, \nu\rangle \mathrm{~d}\mu_g .
	\end{align}
	Similarly, by \eqref{equ:H-minkowski-w0} and Newton-Maclaurin inequality \eqref{equ:N-Mneq}, we have
	\begin{align}\label{equ:kkk3.7}
		&\int_{\Sigma} c_j(\bar{u}) H_{j-1}^F\left(F(\nu)+\omega_0\left\langle \nu, E_{n+1}^F\right\rangle\right)
		\mathrm{~d}\mu_g
         \leq
		\int_{\Sigma} c_{j}(\bar{ u})H_{j}^F\langle X, \nu\rangle \mathrm{~d}\mu_g
		\nonumber
		\\
		&\leq
		\int_{\Sigma} c_{j}(\bar{ u})H_{j-1}^FH_1^F\langle X, \nu\rangle \mathrm{~d}\mu_g.
	\end{align}
	Adding \eqref{equ:kkk3.6} and \eqref{equ:kkk3.7} together and then summing over $j$, using \eqref{equ:3.5}, we
	have
	\begin{align}\label{equ:kkk3.8}
		\int_{\Sigma} \sum_{j=1}^{k}
		\left(b_j(\bar{u}) H_{j-1}^F+c_j(\bar{u}) H_{j-1}^F\right)
		\left(F(\nu)+\omega_0\left\langle \nu, E_{n+1}^F\right\rangle\right)
		\mathrm{~d}\mu_g
		\leq
		\int_{\Sigma} \langle X, \nu\rangle \mathrm{~d}\mu_g.
	\end{align}
	Multiplying the Newton-Maclaurin
	inequality $H^F_1 H^F_{j-1} \geq H^F_j$ by
	$b_j (\bar{ u})$ and summing over $j$ gives
	\begin{align*}
		H^F_1\sum_{j=1}^{k}b_j(\bar{ u})H^F_{j-1}
		\geq
		\sum_{j=1}^{k}b_j(\bar{ u})H^F_{j}.
	\end{align*}
Combining this with \eqref{equ:kkk3.8}, we obtain the inequality
\begin{align*}
	\int_{\Sigma} \langle X, \nu\rangle \mathrm{~d}\mu_g
	\geq&
	\int_{\Sigma}
	 \frac{1}{H^F_1}\sum_{j=1}^{k}
	\left(b_j(\bar{u}) H_{j}^F+c_j(\bar{u}) H_1^F H_{j-1}^F\right)
	\left(F(\nu)+\omega_0\left\langle \nu, E_{n+1}^F\right\rangle\right)
	\mathrm{~d}\mu_g
	\\
	=& \int_{\Sigma}
	\frac{1}{H^F_1}
	\left(F(\nu)+\omega_0\left\langle \nu, E_{n+1}^F\right\rangle\right)
	\mathrm{~d}\mu_g.
\end{align*}
However, anisotropic  Heintze-Karcher inequality	is the reverse inequality \eqref{equ:H_k-wo}. These two inequalities imply the equality in Heintze-Karcher inequality. We conclude that $\Sigma$ is an $\omega_0$-capillary Wulff shape.
\end{proof}

\begin{proof}[\textbf{Proof of Theorem \ref{Thm-Alexan-H>c>H-w0}}]
	(1) For $1<s\leq n$ and  $c'(t)\leq 0$, we have $\frac{d}{dt}(c^{s-1}(t))\leq 0$.
	By \eqref{equ:thm-H>c>H-w0} and Newton-Maclaurin inequality \eqref{equ:N-Mneq}, we know
	$$H^F_1\geq (H^F_{s-1})^{\frac{1}{s-1}}\geq c(\bar{ u}), \ c^s(\bar{ u})\geq H^F_s. $$
	Combining with Corollary \ref{cor:minkowski neq}, we obtain
	\begin{align*}
		&\int_{\Sigma}
		H_{s-1}^F \left(F(\nu)+\omega_0 \<\nu,E_{n+1}^F\>\right)
		\mathrm{~d}\mu_g
		\geq \int_{\Sigma}
		c^{s-1}(\bar{ u})\left(F(\nu)+\omega_0 \<\nu,E_{n+1}^F\>\right)
		\mathrm{~d}\mu_g
		\\
		\overset{\eqref{equ:H-minkowski2-w0}}{\geq}&
		\int_{\Sigma}
		c^{s-1}(\bar{ u})H^F_{1}\<X,\nu\>
		~d\mu_g
		\geq
		\int_{\Sigma}
		c^s(\bar{ u})\<X,\nu\>
		\mathrm{~d}\mu_g
		\geq
		\int_{\Sigma}
		H^F_{s}\<X,\nu\>
		\mathrm{~d}\mu_g
		\\
		=&
		\int_{\Sigma}
		H^F_{s-1}
		\left(F(\nu)+\omega_0 \<\nu,E_{n+1}^F\>\right)
		~d\mu_g.
	\end{align*}
	Thus
	\begin{align*}
		\int_{\Sigma}
		\left(H^F_{s-1}
		-
		c^{s-1}(\bar{ u})\right)
		\left(F(\nu)+\omega_0 \<\nu,E_{n+1}^F\>\right)
		\mathrm{~d}\mu_g
		= 0,
\end{align*}
\begin{align*}
\int_{\Sigma}
		\left(c^{s}(\bar{ u})
		-H^F_{s}\right)\<X, \nu\>
		\mathrm{~d}\mu_g
		= 0,
\end{align*}
and
\begin{align*}
		\int_{\Sigma}
		\left(c^{s-1}(\bar{ u})H^F_{1}
		-
		c^{s}(\bar{ u})\right)\<X,\nu\>
		\mathrm{~d}\mu_g
		= 0.
	\end{align*}
	Combining with $H^F_{s-1}
	-
	c^{s-1}(\bar{ u})\geq 0$, $c^s(\bar{ u})\geq H^F_s$ and $c^{s-1}(\bar{ u})(H^F_{1}-c(\bar{ u}))\<X,\nu\>\geq 0$, we obtain $(H^F_{s-1})^{\frac1{s-1}}=c(\bar{ u})= (H^F_{s})^{\frac1{s}}$ and $c^{s-1}(\bar{ u})(H^F_{1}-c(\bar{ u}))\<X,\nu\>= 0$, which implies $H^F_1=(H^F_{s-1})^{\frac{1}{s-1}}=(H^F_{s})^{\frac1{s}}=c(\bar{ u})$. It follows by Proposition \ref{Prop:Newton-Maclaurin} that, $\Sigma$ is an  $\omega_0$-capillary Wulff shape.
	
	(2) For $c'(\bar{ u})\geq 0$. By \eqref{equ:thm-H/H>c>H/H-w0} and Corollary \ref{cor:minkowski neq}, we have
	\begin{align*}
		\int_{\Sigma}
		H_{s-2}^F
		\left(F(\nu)+\omega_0 \<\nu,E_{n+1}^F\>\right)
		\mathrm{~d}\mu_g
		=\int_{\Sigma}
		H^F_{s-1}\<X,\nu\>
		\mathrm{~d}\mu_g
		\geq
		\int_{\Sigma}
		c(\bar{ u})H^F_{s}\<X,\nu\>
		\mathrm{~d}\mu_g
		\\
		\overset{\eqref{equ:H-minkowski-w0}}{\geq}
		\int_{\Sigma}
		c(\bar{ u})H^F_{s-1}
		\left(F(\nu)+\omega_0 \<\nu,E_{n+1}^F\>\right)
		\mathrm{~d}\mu_g
		\geq
		\int_{\Sigma}
		H^F_{s-2}\left(F(\nu)+\omega_0 \<\nu,E_{n+1}^F\>\right)
		\mathrm{~d}\mu_g,
	\end{align*}
	then
	\begin{align*}
		\int_{\Sigma}
		H^F_{s-1}\<X,\nu\>
		-
		c(\bar{ u})H^F_{s}\<X,\nu\>
		~d\mu_g
		= 0,
	\end{align*}
	combining with $(H^F_{s-1}
	-
	c(\bar{ u})H^F_{s})\<X,\nu\>\geq 0$, we obtain $(H^F_{s-1}
	-
	c(\bar{ u})H^F_{s})\<X,\nu\>=0$. That is $H^F_{s-1}
	=
	c(\bar{ u})H^F_{s}$, $s>1$. By Theorem \ref{Thm-Alexan-a=b-w0}(1), we know $\Sigma$ is an  $\omega_0$-capillary Wulff shape.
\end{proof}


 \begin{proof}[\textbf{Proof of Theorem \ref{thm1.10}}]
The anisotropic $k$-convexity assumption says $H^F_j>$ 0 for all $j =0, \cdots, k$. It follows from  \eqref{equ:kkk4.2}, \eqref{equ:u}  and Proposition \ref{prop2.1} that $\<X,\nu\>>0$ for $X\in\Sigma$.

Assume first that $k \geq 2$. By Newton-Maclaurin inequality \eqref{equ:N-Mneq}, we have for $0 \leq i<j \leq k$,
\begin{align}\label{equ:kkk4.4}
	\frac{1}{H^F_1}\leq
\left(\frac{H^F_i}{H^F_j}\right)^{\frac{1}{j-i}}
\leq \frac{H^F_{j-1}}{H^F_j} .
\end{align}
Therefore, by Newton-Maclaurin inequality \eqref{equ:N-Mneq} again,
\begin{align}\label{equ:kkk4.5}
	\beta\bar{ u}=\sum_{i<j} a_{i, j}\left(\frac{H^F_i}{H^F_j}\right)^{\frac{1}{j-i}} \leq \sum_{i<j} a_{i, j} \frac{H^F_{j-1}}{H^F_j} \leq \sum_{i<j} a_{i, j} \frac{H^F_{k-1}}{H^F_k}=\frac{H^F_{k-1}}{H^F_k},
\end{align}
and
\begin{align}\label{equ:kkk4.6}
	\beta\bar{ u}=\sum_{i<j} a_{i, j}\left(\frac{H^F_i}{H^F_j}\right)^{\frac{1}{j-i}} \geq \sum_{i<j} a_{i, j} \frac{1}{H^F_1}=\frac{1}{H^F_1} .
\end{align}
The inequality \eqref{equ:kkk4.5} implies
$$
\beta  \int_{\Sigma} H^F_k \<X,\nu\> \mathrm{~d}\mu_g
\leq \int_{\Sigma} H^F_{k-1}\left(F(\nu)+\omega_0\left\langle \nu, E_{n+1}^F\right\rangle\right)\mathrm{~d}\mu_g,
$$
which in turn implies $\beta \leq 1$ by the anisotropic Hsiung-Minkowski formula \eqref{equ:Thm1.2}.
On the other hand, \eqref{equ:kkk4.6} implies
$$
\beta  \int_{\Sigma} H^F_1 \<X,\nu\> \mathrm{~d}\mu_g
\geq \int_{\Sigma} H^F_{0}\left(F(\nu)+\omega_0\left\langle \nu, E_{n+1}^F\right\rangle\right)\mathrm{~d}\mu_g,
$$
and hence $\beta \geq 1$ again by the anisotropic Hsiung-Minkowski formula \eqref{equ:Thm1.2}.

We conclude that $\beta=1$ and all the inequalities in \eqref{equ:kkk4.4}  are equalities. Therefore $\Sigma$ is  anisotropic umbilical and so is an $\omega_0$-capillary Wulff shape.

When $k=1$, \eqref{equ:kkk4.6}  becomes an equality and hence $\beta=1$ by \eqref{equ:Thm1.2}. By the Newton-Maclaurin inequality, we have
\begin{align}\label{equ:kkk4.7}
	H^F_2 \bar{ u}=\frac{H^F_2}{H^F_1} \leq \frac{H^F_1}{H^F_0}=H^F_1.
\end{align}
Multiplying this inequality with $F(\nu)+\omega_0\left\langle \nu, E_{n+1}^F\right\rangle$,
integrating and comparing with the anisotropic Hsiung-Minkowski formula \eqref{equ:Thm1.2} for $k=1$,
we again deduce that \eqref{equ:kkk4.7} is an equality,  hence $\Sigma$ is an $\omega_0$-capillary Wulff shape.
 \end{proof}

\section{Reference}


\end{document}